\documentclass[12pt]{amsart}
\usepackage{graphicx}
\usepackage[headings]{fullpage}
\usepackage{amssymb,epic,eepic,epsfig,amsbsy,amsmath,amscd}

\numberwithin{equation}{section}
                        \theoremstyle{plain}
\usepackage{mathrsfs}

\newcommand\no[1]{}

\newtheorem{theorem}{Theorem}[section]

\newtheorem{lemma}[theorem]{Lemma}

\newtheorem{proposition}[theorem]{Proposition}

\theoremstyle{definition}

\def\BZ{\mathbb Z}

\def\ve{\varepsilon}
\def\be { \begin{equation} }
\def\ee { \end{equation} }

\begin{document}

\title[On the braid index of a two-bridge knot]
{On the braid index of a two-bridge knot}

\begin{abstract}
In this paper, we consider two properties on the braid index of a two-bridge knot. 
We prove an inequality on the braid indices of two-bridge knots if there exists an epimorphism between their knot groups. 
Moreover, we provide the average braid index of all two-bridge knots with a given crossing number. 
\end{abstract}

\author{Masaaki Suzuki}
\address{Department of Frontier Media Science,
  Meiji University,
  4-21-1 Nakano, Nakano-ku, Tokyo, 164-8525, Japan}
\email{mackysuzuki@meiji.ac.jp}

\author{Anh T. Tran}
\address{Department of Mathematical Sciences,
The University of Texas at Dallas,
Richardson, TX 75080-3021, USA}
\email{att140830@utdallas.edu}

\thanks{2020 {\it Mathematics Subject Classification}.
Primary 57K10, 57K31.}

\thanks{{\it Key words and phrases.\/}
braid index, crossing number, epimorphism, minimal knot, two-bridge knot.}

\maketitle

\section{Introduction}

It is known as Alexander theorem that each knot can be represented as the closure of a braid. 
The minimal number of strings to represent a knot $K$ in this way is called the braid index of $K$, which is denoted by ${\rm braid} (K)$. 
For a non-trivial knot $K$, it is easy to see that ${\rm braid} (K) \geq 2$. 
In this paper, we consider two properties on the braid index of a two-bridge knot. 

First, we consider the relationship between the braid indices of two two-bridge knots if there exists an epimorphism between their knot groups. 
Let $G(K)$ be the knot group of $K$, namely, the fundamental group of the exterior of $K$ in $S^3$. 
We write $K \geq K'$ if there exists an epimorphism $\varphi : G(K) \to G(K')$. 
Boileau, Boyer, Reid, and Wang in \cite{BBRW} showed that if $K$ is a two-bridge knot and $K \geq K'$, then $K'$ is also a two-bridge knot or the trivial knot. 
In the previous papers \cite{SZK} and \cite{ST1}, it was shown that 
if $K$ is a two-bridge knot and $K \geq K'$ then the following inequalities on the crossing numbers $c(K),c(K')$ and the genera $g(K),g(K')$ hold:
\begin{align*}
& c(K) \geq 3 \, c(K'), \\
& g(K) \geq 3 \, g(K') -1.
\end{align*} 
In this paper, we prove a similar inequality on the braid indices ${\rm braid} (K), {\rm braid} (K')$: 
\[
{\rm braid} (K) \geq  3 \, {\rm braid}(K') - 4 .
\]
As a corollary, we show that if ${\rm braid} (K) \leq 4$, then $K'$ is a torus knot.  

Second, we consider the average braid index of two-bridge knots with respect to crossing number. 
In the previous paper \cite{ST2}, the average genus $\overline{g}_{c}$ of all the two-bridge knots with crossing number $c$ is given by: 
\[
\overline{g}_{c} =
\left\{
\begin{array}{ll}
\displaystyle{ \frac{3c+1}{12} +  \frac{c-5}{2^c - 4}} & c \equiv 0 \pmod{2} 
\smallskip \\
\displaystyle{ \frac{3c+1}{12}+  \frac{1}{3 \cdot 2^{(c-3)/2}}} & c \equiv 1 \pmod{4} 
\smallskip\\
\displaystyle{ \frac{3c+1}{12} +  \frac{2^{(c+1)/2} - 3 c + 11}{12 \left(2^{c-3} + 2^{(c-3)/2} + 1\right)}} & c \equiv 3 \pmod{4}
\end{array}
\right. 
\]
for $c \geq 3$.
In particular, $\overline{g}_c \sim \displaystyle{\frac{c}{4}  + \frac{1}{12}}$ as $c \to \infty$. 
Similarly, in this paper we show that the average braid index of all the two-bridge knots with crossing number $c$ is given by:
\[
\overline{{\rm braid}}_c =
\left\{
\begin{array}{ll}
\displaystyle{\frac{3c+11}{9}  + \frac{2c-4}{3(2^{c-2}-1)}} & c \equiv 0 \pmod{2} 
\smallskip \\
\displaystyle{\frac{3c+11}{9}   - \frac{2^{(c+3)/2} + 9 c - 19}{9 (2^{c-2} + 2^{(c-1)/2})}} & c \equiv 1 \pmod{4} 
\smallskip\\
\displaystyle{\frac{3c+11}{9}   - \frac{2^{(c+3)/2} + 3 c - 5}{9 (2^{c-2} + 2^{(c-1)/2}+2)}} & c \equiv 3 \pmod{4}
\end{array}
\right.
\]
for $c \geq 3$. 
In particular,  $\overline{{\rm braid}}_c \sim \displaystyle{\frac{c}{3} + \frac{11}{9}}$ as $c \to \infty$.

The paper is organized as follows. In Section \ref{prelim} we review two-bridge knots, their crossing numbers and braid indices, and the Ohtsuki-Riley-Sakuma construction of epimorphisms between two-bridge knot groups. In Section \ref{sect_epi} we prove the inequality on the braid indices of two-bridge knots if there exists an epimorphism between their knot groups. In Section \ref{sect_smallbraid} we determine the minimality of two-bridge knots with small braid index. Finally, in Section \ref{avg} we compute the average braid index of all the two-bridge knots with a given crossing number.

\section{Two-bridge knots and their braid indices} \label{prelim}

In this section, we review some properties of two-bridge knots. 
Especially, we recall the braid index of a two-bridge knot and 
the Ohtsuki-Riley-Sakuma construction of epimorphisms between two-bridge knot groups. 

It is known that a two-bridge knot corresponds to a rational number. 
Then $K(r)$ stands for the two-bridge knot corresponding to a rational number $r$. 
Furthermore, a rational number $r$ can be expressed as a continued fraction 
\[
 r = 
 [a_1,a_2, \ldots, a_{m-1},a_m] = 
\dfrac{1}{a_1 + \dfrac{1}{a_2 + \dfrac{1}{\ddots \dfrac{1}{a_{m-1} + \dfrac{1}{a_m}}}}} .
\]
Moreover, every two-bridge knot is represented by a reduced even continued fraction, 
that is a continued fraction of the form $[2 a_1, 2 a_2, . . . , 2 a_{2m}]$ where each $a_i \neq 0$. 
It is known that $K([2a_1,2a_2, \ldots, 2a_{2m}])$ and $K([-2a_{2m}, \ldots,- a_2, -2a_1])$ are equivalent  
and that $K([-2a_1,-2a_2, \ldots, -2a_{2m}]) =K([2a_{2m},\ldots,2 a_2, 2a_1])$ is their mirror image. 

Let $[\mathbf{a}] = [2a_1, 2a_2, \cdots, 2a_{2m}]$ be a reduced even continued fraction corresponding to  $K = K([\mathbf{a}])$. 
Then the crossing number $c(K)$ and the braid index $\text{braid}(K)$ of $K$ are   
\begin{eqnarray*}
c(K) &=& \left(\sum_{i=1}^{2m} 2|a_i| \right) -  t(\mathbf{a}), \\
\text{braid}(K) &=& \left(\sum_{i=1}^{2m} |a_i| \right) -  t(\mathbf{a}) +1,
\end{eqnarray*}
where $t(\mathbf{a})$ is the number of sign changes in the sequence $\mathbf{a} = (2 a_1, 2 a_2, \ldots, 2 a_{2m})$. This implies that $\mathrm{braid}(K) = \frac{1}{2} c(K) +1 - \frac{1}{2}  t(\mathbf{a})$. 
Note that $t(\mathbf{a}) \le 2m-1$.

All epimorphisms between two-bridge knot groups are completely characterized by the following distinguished theorem,  
which plays an important role in Sections \ref{sect_epi} and \ref{sect_smallbraid}. 

\begin{theorem}
[Ohtsuki-Riley-Sakuma \cite{ORS}, Agol \cite{agol}, Aimi-Lee-Sakuma \cite{ALS}]\label{thm:ors}
Let $K(r), K(\tilde{r})$ be $2$-bridge knots, 
where $r = [a_1,a_2,\ldots,a_m]$. 
There exists an epimorphism $\varphi : G(K(\tilde{r})) \to G(K(r))$ 
if and only if $\tilde{r}$ can be written as 
\[
 \tilde{r} = 
[\varepsilon_1 {\bf a}, 2 c_1, 
\varepsilon_2 {\bf a}^{-1}, 2 c_2, 
\varepsilon_3 {\bf a}, 2 c_3, 
\varepsilon_4 {\bf a}^{-1}, 2 c_4, 
\ldots, 
\varepsilon_{2n} {\bf a}^{-1}, 2 c_{2n}, \varepsilon_{2n+1} {\bf a}] , 
\]
where ${\bf a} = (a_1, a_2,\ldots,a_m), {\bf a}^{-1} = (a_m, a_{m-1},\ldots,a_1)$, 
$\varepsilon_i = \pm 1 \, \, (\varepsilon_1 = 1)$, $c_i \in {\mathbb Z}$, and $n \in {\mathbb N}$. 
\end{theorem}

\section{Braid index and epimorphism}\label{sect_epi}

In this section, we show the following theorem, which is one of the main theorems of this paper. 

\begin{theorem} \label{ineq}
Let $K$ be a two-bridge knot and $K'$ a knot. 
If $K \geq K'$, then we have 
\[
{\rm braid} (K) \geq  3 \, {\rm braid}(K') - 4 .
\]
\end{theorem}

\begin{proof}
Let $K'= K([\mathbf{a}])$, where $\mathbf{a} = (2a_1, 2a_2, \cdots, 2a_{2m})$ is  a reduced even continued fraction. 
Then the braid index of $K'$ is 
$$
\text{braid}(K') = \sum_{i=1}^{2m}|a_i| - t(\mathbf{a}) + 1.
$$
By Theorem \ref{thm:ors}, the sequence of a reduced even continued fraction for $K = K([\tilde{{\mathbf a}}])$ has the form 
\[
\tilde{\mathbf{a}} = (\ve_1 \mathbf{a}, 2c_1, \ve_2 \mathbf{a}^{-1}, 2c_2, \ve_3 \mathbf{a}, 2c_3, 
\ve_4 \mathbf{a}^{-1}, 2c_4, \cdots, \ve_{2r} \mathbf{a}^{-1}, 2c_{2r}, \ve_{2r+1} \mathbf{a}) .
\]
Here $\ve_1=1$ and $\ve_j = \pm 1$, and $c_j \in \BZ$. Note that if $c_j=0$ then we can assume $\ve_{j+1} = \ve_j$ (see \cite{SZK} for details). 
To get a reduced even continued fraction, we delete $c_j=0$ and combine $\ve_j \mathbf{a}^{(-1)^{j-1}}$ with $\ve_{j+1} \mathbf{a}^{(-1)^j} = \ve_{j} \mathbf{a}^{(-1)^j}$. 
Hence the braid index of $K$ is 
$$
\text{braid}(K) = (2r+1)\sum_{i=1}^{2m}|a_i| + \sum_{\substack{j=1 \\ c_j \not= 0}}^{2r} |c_j| - t(\tilde{\mathbf{a}}) + 1.
$$
Note that $t(\tilde{\mathbf{a}}) \le (2r+1) t(\mathbf{a}) + 2 \#\{c_j \not= 0\}$. 
We have
\begin{eqnarray*}
&&\text{braid}(K) \\
&=& (2r+1)\left( \sum_{i=1}^{2m}|a_i| -  t(\mathbf{a}) \right) + \sum_{\substack{j=1 \\ c_j \not= 0}}^{2r} |c_j|  - 2 \#\{c_j \not= 0\}+ 1 \\
&& + \,  \left( (2r+1) t(\mathbf{a}) + 2 \#\{c_j \not= 0\} - t(\tilde{\mathbf{a}})\right) \\
&=& (2r+1) (\text{braid}(K') - 1)  + \sum_{\substack{j=1 \\ c_j \not= 0}}^{2r} (|c_j|  - 1) - \#\{c_j \not= 0\} + 1 \\
&& + \,  \left( (2r+1) t(\mathbf{a}) + 2 \#\{c_j \not= 0\} - t(\tilde{\mathbf{a}})\right) \\
&=& (2r+1) (\text{braid}(K') - 2)  + \sum_{\substack{j=1 \\ c_j \not= 0}}^{2r} (|c_j|  - 1) + (2r - \#\{c_j \not= 0\}) + 2\\
&& + \,  \left( (2r+1) t(\mathbf{a}) + 2 \#\{c_j \not= 0\} - t(\tilde{\mathbf{a}})\right) \\
&=& 3 (\text{braid}(K') - 2) + 2 + (2r-2) (\text{braid}(K') - 2)  + \sum_{\substack{j=1 \\ c_j \not= 0}}^{2r} (|c_j|  - 1) + (2r - \#\{c_j \not= 0\}) \\
&& + \,  \left( (2r+1) t(\mathbf{a}) + 2 \#\{c_j \not= 0\} - t(\tilde{\mathbf{a}})\right).
\end{eqnarray*}
This implies the desired inequality 
\begin{equation}\label{mainineq}
\text{braid}(K) \ge  3 (\text{braid}(K') - 2) + 2 = 3 \, \text{braid}(K') - 4,
\end{equation} 
since the other terms are non-negative: 
\begin{eqnarray}
(2r-2) (\text{braid}(K') - 2) &\ge& 0, \label{ineq1}\\
\sum_{\substack{j=1 \\ c_j \not= 0}}^{2r} (|c_j|  - 1) &\ge& 0, \label{ineq2}\\ 
2r - \#\{c_j \not= 0\} &\ge& 0, \label{ineq3}\\
(2r+1) t(\mathbf{a}) + 2 \#\{c_j \not= 0\} - t(\tilde{\mathbf{a}})&\ge& 0. \label{ineq4}
\end{eqnarray}
\end{proof}

\section{Minimality of two-bridge knots with small braid index}\label{sect_smallbraid}

We call a knot $K$ minimal if $K \geq K'$ implies $K' = K$ or $K'$ is the trivial knot. 
In this section, we discuss when a two-bridge knot $K$ with $\text{braid}(K) \leq 4$ is minimal. 
Suppose that $K \geq K'$ and that $K'$ is not the trivial knot. Then $K'$ is also a two-bridge knot. 
Since $3 \, \text{braid}(K') - 4 \le \text{braid}(K) \le 4$, we obtain $\text{braid}(K')=2$. 

\begin{lemma} \label{braid2}
Let $K$ be a two-bridge knot. If ${\rm braid}(K)=2$, then 
$K=K([\underbrace{2, -2,  \cdots, 2, -2}_{2m}])$ which is the torus knot $T(2m+1, 2)$. 
\end{lemma}

\begin{proof}
We first note that for $K = K([\mathbf{a}])$ we have
$$
\text{braid}(K) = \sum_{i=1}^{2m}|a_i| - t(\mathbf{a}) + 1 =  \sum_{i=1}^{2m} (|a_i| -1) + (2m-1 - t(\mathbf{a})) + 2 \ge 2,
$$
where $\mathbf{a} = (2a_1, 2a_2, \cdots, 2a_{2m})$ and each $a_i \not= 0$.
Equality holds if and only if $|a_i|=1$ and $t(\mathbf{a}) = 2m-1$. 
This means that $|a_i|=1$ and $a_i a_{i+1}<0$ for any $i$, i.e. $\mathbf{a} = (\underbrace{2, -2,  \cdots, 2, -2}_{2m})$. Hence $K$ is the torus knot $T(2m+1, 2)$.
\end{proof}

Note that all non-trivial knots with braid index $2$ are two-bridge knots, 
since it is well known that the bridge number is at most the braid index. 

By Lemma \ref{braid2}, in order to see if a two-bridge knot $K$ with $\text{braid}(K) \leq 4$ is minimal or not, 
we only need to check whether $K \geq T(2m+1,2)$. 
Namely, we consider $K= K([\tilde{\mathbf{a}}])$ in the case  
$$
\tilde{\mathbf{a}} = (\ve_1 \mathbf{a}, 2c_1, \ve_2 \mathbf{a}^{-1}, 2c_2, \ve_3 \mathbf{a}, 2c_3, 
\ve_4 \mathbf{a}^{-1}, 2c_4, \cdots, \ve_{2r} \mathbf{a}^{-1}, 2c_{2r}, \ve_{2r+1} \mathbf{a}).
$$
where $\mathbf{a} = (\underbrace{2, -2,  \cdots, 2, -2}_{2m})$, $\ve_1=1$ and $\ve_j = \pm 1$, and $c_j \in \BZ$. 
Note that if $c_j=0$ then we can assume $\ve_{j+1} = \ve_j$.

\subsection{The braid index of $K$ is $2$} 

Let $K$ be a two-bridge knot with braid index $2$, which is a torus knot by Lemma \ref{braid2}.  
Assume that $K \geq K'$ and that $K'$ is not the trivial knot. 
Then the braid index of $K'$ is also $2$, in particular, $K' = K([\mathbf{a}])$ where $\mathbf{a} = (\underbrace{2, -2,  \cdots, 2, -2}_{2m})$.
Since the equality holds in (\ref{mainineq}) and the left hand side of (\ref{ineq1}) is zero, 
the left hand sides of (\ref{ineq2}), (\ref{ineq3}), (\ref{ineq4}) are all zero. 
Then we have $|c_j|=1$ and $(-1)^{j}c_j < 0$ for all $j$, that is, 
$\tilde{\mathbf{a}} =  (\mathbf{a}, 2, \mathbf{a}^{-1}, -2, \mathbf{a}, 2, \cdots, \mathbf{a}^{-1}, -2, \mathbf{a})$. 
Since the length of $\tilde{\mathbf{a}}$ is equal to $(2r+1)(2m) + 2r$, 
we see that $K=K([\tilde{\mathbf{a}}])$ is the torus knot $T((2r+1)(2m)+2r +1, 2)= T((2r+1)(2m+1), 2)$. 
Therefore we obtain the following proposition. 

\begin{proposition} \label{braid2minimal} 
A two-bridge knot $K$ with $\emph{braid}(K)=2$ is minimal if and only if 
$K=K([\underbrace{2, -2, \cdots, 2, -2}_{2k}])$
(the torus knot $T(2k+1, 2)$) where $2k+1$ is a prime number.
\end{proposition}

\subsection{The braid index of $K$ is $3$} 
\begin{lemma} \label{braid3}
Let $K$ be a two-bridge knot. If $\emph{braid}(K)=3$, then one of the following holds
\begin{itemize}
\item[$-$] (type $3a$) $K=K([s_1 n_1, s_2 n_2, \cdots, s_{2m} n_{2m}])$
where $s_i = (-1)^{i-1}$, and there exists $i_0$ such that $n_{i_0} = 4$ and $n_i=2$ for $i \not= i_0$. 
\item[$-$]  (type $3b$)  $K=K([2s_1, 2s_2, \cdots, 2s_{2m}])$
where there exists $i_0$ such that $s_i = (-1)^{i-1}$ for $i \le i_0$ and $s_i = (-1)^{i}$ for $i \ge i_0+1$.
\end{itemize}
\end{lemma}

\begin{proof}
Recall that for $K = K([\mathbf{a}])$ we have
$$
\text{braid}(K) = \sum_{i=1}^{2m}|a_i| - t(\mathbf{a}) + 1 =  \sum_{i=1}^{2m} (|a_i| -1) + (2m-1 - t(\mathbf{a})) + 2
$$
where $\mathbf{a} = (2a_1, 2a_2, \cdots, 2a_{2m})$ and each $a_i \not= 0$.
Since $\mathrm{braid}(K)=3$, we get
$$
\sum_{i=1}^{2m} (|a_i| -1) + (2m-1 - t(\mathbf{a})) =1.
$$

Since $|a_i| \ge 1$ and $t(\mathbf{a}) \le 2m-1$, there are 2 cases to consider.

\underline{Case 1}: $\sum_{i=1}^{2m} (|a_i| -1) =1$ and $2m-1 - t(\mathbf{a})=0$. Then there exists $i_0$ such that $|a_{i_0}| = 2$, and $|a_i|=1$ for $i \not= i_0$. Moreover $a_i a_{i+1}<0$ for all $i$ (since $t(\mathbf{a}) =  2m-1$). Hence $a_{i_0} = 2(-1)^{i_0-1}$, and $a_{i} = (-1)^{i-1}$ for $i \not= i_0$. 

\underline{Case 2}: $\sum_{i=1}^{2m} (|a_i| -1) =0$ and $2m-1 - t(\mathbf{a})=1$. Then $|a_{i}| =1$ for all $i$. Moreover, there exists $i_0$ such that $a_{i_0} a_{i_0+1}>0$, and $a_i a_{i+1}<0$ for all $i \not= i_0$ (since $t(\mathbf{a}) =  2m-2$). Hence $a_i = (-1)^{i-1}$ for $i \le i_0$ and $a_i = (-1)^{i}$ for $i \ge i_0+1$.
\end{proof}

Let $K$ be a two-bridge knot with braid index $3$. 
Assume that $K \geq K'$ and that $K'$ is not the trivial knot.
Since the left hand side of (\ref{ineq1}) is zero, we have 
\begin{eqnarray*}
&& \sum_{\substack{j=1 \\ c_j \not= 0}}^{2r} (|c_j|  - 1) + (2r - \#\{c_j \not= 0\}) +  \left( (2r+1) t(\mathbf{a}) + 2 \#\{c_j \not= 0\} - t(\tilde{\mathbf{a}})\right) \\
&=& \text{braid}(K) -2 =1, 
\end{eqnarray*}
namely,  
one of the left hand sides of (\ref{ineq2}), (\ref{ineq3}), (\ref{ineq4}) is $1$ and that the others are zero. 
We consider the following three cases: $(A,B,C) = (1,0,0), (0,1,0), (0,0,1)$,  
where $A,B,C$ are the left hand sides of (\ref{ineq2}), (\ref{ineq3}), (\ref{ineq4}) respectively.  

\underline{Case 1}: $\sum_{\substack{j=1 \\ c_j \not= 0}}^{2r} (|c_j|  - 1)=1$, $2r - \#\{c_j \not= 0\}=0$ and $(2r+1) t(\mathbf{a}) + 2 \#\{c_j \not= 0\} - t(\tilde{\mathbf{a}}) =0$. Then there exists $j_0$ such that $|c_{j_0}|=2$ and $|c_j| = 1$ for $j \not= j_0$. 
Moreover $\ve_{j+1} = \ve_j$ and $(-1)^j \ve_{j} c_{j} <0$ for all $j$. Hence $\ve_1 = \cdots = \ve_{2r+1} = 1$, and $(-1)^{j-1}c_{j} >0$ for $j$. In this case $K$ is of type $3a$ in Lemma \ref{braid3} with 
\begin{itemize}
\item $i_0 = j_0 (2m+1)$, and 
\item length $2k = (2r+1)2m+2r$, so $2k + 1= (2r+1)(2m+1)$.
\end{itemize}

\underline{Case 2}: $\sum_{\substack{j=1 \\ c_j \not= 0}}^{2r} (|c_j|  - 1)=0$, $2r - \#\{c_j \not= 0\}=1$ and $(2r+1) t(\mathbf{a}) + 2 \#\{c_j \not= 0\} - t(\tilde{\mathbf{a}}) =0$. 
Then there exists $j_0$ such that $c_{j_0}=0$ (note that $\ve_{j_0+1} = \ve_{j_0}$) and $|c_j| = 1$ for $j \not= j_0$. 
Moreover $\ve_{j+1} = \ve_j$ and $(-1)^j \ve_{j} c_{j} <0$ for all $j \not= j_0$. Hence $\ve_1 = \cdots = \ve_{2r+1} = 1$, and $(-1)^{j-1} c_{j} >0 $ for $j \not= j_0$. In this case, by deleting $c_{j_0}=0$ and combining $\ve_{j_0} \mathbf{a}^{(-1)^{j_0-1}}$ with $\ve_{j_0+1} \mathbf{a}^{(-1)^{j_0}} = \ve_{j_0} \mathbf{a}^{(-1)^{j_0}}$, we see that $K$ is of type $3a$ in Lemma \ref{braid3} with 
\begin{itemize}
\item $i_0 = j_0 (2m+1)-1$, and  
\item length $2k = (2r+1)2m+2r-2$, so $2k + 3= (2r+1)(2m+1)$.
\end{itemize}

\underline{Case 3}: $\sum_{\substack{j=1 \\ c_j \not= 0}}^{2r} (|c_j|  - 1)=0$, $2r - \#\{c_j \not= 0\}=0$, and $(2r+1) t(\mathbf{a}) + 2 \#\{c_j \not= 0\} - t(\tilde{\mathbf{a}}) =1$. Then $|c_j|=1$ for all $j$. Moreover, there exists $j_0$ such that $\ve_{j_0+1} = - \ve_{j_0}$, and $\ve_{j+1} =\ve_{j}$ and $(-1)^j \ve_{j} c_{j} <0$ for $j \not= j_0$. This implies that $\ve_j = 1$ for $j \le j_0$, and $\ve_j = -1$ for $j \ge j_0 +1$. Hence 
$(-1)^{j-1} c_j > 0$ for $j < j_0$, and $(-1)^{j} c_j >0$ for $j > j_0$. In this case $K$ is of type $3b$ in Lemma \ref{braid3} with 
\begin{itemize}
\item  $i_0 = j_0 (2m+1)$ or $i_0 = j_0 (2m+1) -1$ (depending on whether $(-1)^{j_0 - 1} c_{j_0} >0$ or $(-1)^{j_0} c_{j_0} >0$), and
\item length $2k = (2r+1)2m+2r$, so $2k + 1= (2r+1)(2m+1)$.
\end{itemize}

Therefore we obtain the following proposition. 

\begin{proposition} \label{braid3minimal}
Let $K$ be a two-bridge knot. If $\emph{braid}(K)=3$, then $K$ is not minimal if and only if  one of the following holds
\begin{itemize}
\item[$-$] (type $3A1$) $K=K([s_1 n_1, s_2 n_2, \cdots, s_{2k} n_{2k}])$ where $s_i = (-1)^{i-1}$ and there exists $(r,m,i_0,j_0) \in {\mathbb N}^4$ such that $2k + 1= (2r+1)(2m+1)$, $i_0 = j_0 (2m+1) ~(1 \leq j_0 \leq 2 r)$, $n_{i_0} = 4$, and $n_i=2$ for $i \not= i_0$.

\item[$-$] (type $3A2$) $K=K([s_1 n_1, s_2 n_2, \cdots, s_{2k} n_{2k}])$ where $s_i = (-1)^{i-1}$
and there exists $(r,m,i_0,j_0) \in {\mathbb N}^4$ such that 
$2k+3 = (2r+1)(2m+1)$,  $i_0 = j_0 (2m+1) - 1 ~(1 \leq j_0 \leq 2 r)$, $n_{i_0} = 4$, and $n_i=2$ for $i \not= i_0$.

\item[$-$] (type $3B$) $K=K([2s_1, 2s_2, \cdots, 2s_{2k}])$ where 
there exists $(r,m,i_0,j_0) \in {\mathbb N}^4$ such that 
$2k+1 = (2r+1)(2m+1)$, $i_0 = j_0 (2m+1)$ or $i_0 = j_0 (2m+1) -1 ~(1 \leq j_0 \leq 2 r)$, $s_i = (-1)^{i-1}$ for $i \le i_0$, and $s_i = (-1)^{i}$ for $i \ge i_0+1$.
\end{itemize}
\end{proposition}

\subsection{The braid index of $K$ is $4$} 

\begin{lemma} \label{braid4}
Let $K$ be a two-bridge knot. If $\emph{braid}(K)=4$, then one of the following holds
\begin{itemize}
\item[$-$] (type $4a$)  $K=K([s_1 n_1, s_2 n_2, \cdots, s_{2m} n_{2m}])$
where $s_i = (-1)^{i-1}$, and  there exists $i_0$ such that $n_{i_0} = 6$ and $n_i=2$ for $i \not= i_0$.
\item[$-$]  (type $4b$)  $K=K([s_1 n_1, s_2 n_2, \cdots, s_{2m} n_{2m}])$
where $s_i = (-1)^{i-1}$, and there exist $i_0 \not= i_1$ such that $n_{i_0} = n_{i_1} = 4$ and $n_i=2$ for $i \not= i_0, i_1$.
\item[$-$] (type $4c$)  $K=K([s_1 n_1, s_2 n_2, \cdots, s_{2m} n_{2m}])$
where there exist $i_0, i_1$ (it might happen that $i_0 = i_1$) such that $s_i = (-1)^{i-1}$ for $i \le i_1$ and $s_i = (-1)^{i}$ for $i \ge i_1+1$, $n_{i_0}=4$ and $n_i=2$ for $i \not= i_0$.
\item[$-$] (type $4d$)  $K=K([2s_1, 2s_2, \cdots, 2s_{2m}])$
where there exist $i_0 < i_1$ such that  $s_i = (-1)^{i-1}$ for $i \le i_0$ or $i \ge i_1 +1$, and $s_i = (-1)^{i}$ for $i_0+1 \le i \le i_1$.

\end{itemize}
\end{lemma}

\begin{proof}
Since $\mathrm{braid}(K)=4$ we have
$$
\sum_{i=1}^{2m} (|a_i| -1) + (2m-1 - t(\mathbf{a})) = \mathrm{braid}(K) - 2 = 2
$$
where $\mathbf{a} = (2a_1, 2a_2, \cdots, 2a_{2m})$ and each $a_i \not= 0$. We check the following three cases. 

\underline{Case 1}: $\sum_{i=1}^{2m} (|a_i| -1) =2$ and $2m-1 - t(\mathbf{a})=0$. Then $a_i a_{i+1}<0$ for all $i$. There are 2 subcases to consider. 
\begin{enumerate}
\item There exists $i_0$ such that $|a_{i_0}| = 3$, and $|a_i|=1$ for $i \not= i_0$. Hence $a_{i_0} = 3 (-1)^{i_0-1}$, and $a_i = (-1)^{i-1}$ for $i \not= i_0$.
\item There exist $i_0 \not= i_1$ such that $|a_{i_0}| = |a_{i_1}| =2$, and $|a_i|=1$ for $i \not= i_0, i_1$. Hence $a_{i_0} = 2 (-1)^{i_0-1}$, $a_{i_1} = 2 (-1)^{i_1-1}$, and $a_i = (-1)^{i-1}$ for $i \not= i_0, i_1$.
\end{enumerate}

\underline{Case 2}: $\sum_{i=1}^{2m} (|a_i| -1) =1$ and $2m-1 - t(\mathbf{a})=1$. Then there exists $i_0$ such that $|a_{i_0}|=2$ and  $|a_i|=1$ for $i \not= i_0$. There also exists $i_1$ (it might happen that $i_1 = i_0$) such that $a_{i_1} a_{i_1+1}>0$, and $a_i a_{i+1}<0$ for all $i \not= i_1$. Hence $a_i = (-1)^{i-1}$ for $i \not= i_0$ and $i \le i_1$, and $a_i = (-1)^{i}$ for $i \not= i_0$ and $i \ge i_1+1$. Moreover $a_{i_0} = 2(-1)^{i_0-1}$ if $i_0 \le i_1$, and $a_{i_0} = 2(-1)^{i_0}$ if $i_0 \ge i_1+1$. 

\underline{Case 3}: $\sum_{i=1}^{2m} (|a_i| -1) =0$ and $2m-1 - t(\mathbf{a})=2$. Then $|a_i| =1$ for all $i$. Moreover, there exist $i_0 < i_1$ such that $a_{i_0} a_{i_0+1}>0$, $a_{i_1} a_{i_1+1}>0$, and $a_{i} a_{i+1} < 0$ for all $i \not= i_0, i_1$. Hence $a_i =(-1)^{i-1}$ for $i \le i_0$ or $i \ge i_1 +1$, and $a_i = (-1)^i$ for $i_0 +1 \le i \le i_1$.
\end{proof}

Let $K$ be a two-bridge knot with braid index $4$. 
Assume that $K \geq K'$ and that $K'$ is not the trivial knot.
Since the left hand side of (\ref{ineq1}) is zero, we have  
\begin{eqnarray*}
&& \sum_{\substack{j=1 \\ c_j \not= 0}}^{2r} (|c_j|  - 1) + (2r - \#\{c_j \not= 0\}) +  \left( (2r+1) t(\mathbf{a}) + 2 \#\{c_j \not= 0\} - t(\tilde{\mathbf{a}})\right) \\
&=& \mathrm{braid}(K) - 2 = 2,
\end{eqnarray*}
namely, 
the sum of the left hand sides of (\ref{ineq2}), (\ref{ineq3}), and (\ref{ineq4}) is $2$. 
Then we consider the following six cases: $(A,B,C) = (2,0,0), (1,1,0), (1,0,1), (0,2,0), (0,1,1), (0,0,2)$,  
where $A,B,C$ are the left hand sides of (\ref{ineq2}), (\ref{ineq3}), (\ref{ineq4}) respectively.  

\underline{Case 1}: $\sum_{\substack{j=1 \\ c_j \not= 0}}^{2r} (|c_j|  - 1)=2$, $2r - \#\{c_j \not= 0\}=0$ and $(2r+1) t(\mathbf{a}) + 2 \#\{c_j \not= 0\} - t(\tilde{\mathbf{a}}) =0$. Then  $\ve_{j+1} = \ve_j$ and $(-1)^j \ve_{j} c_{j} <0$ for all $j$. This implies that $\ve_1 = \cdots = \ve_{2r+1} = 1$ and $(-1)^{j-1} c_j > 0$ for all $j$. There are 2 subcases to consider:
\begin{enumerate}
\item There exists $j_0$ such that $|c_{j_0}| = 3$, and $|c_j|=1$ for $j\not= j_0$. In this case $K$ is of type $4a$ in Lemma \ref{braid4} with 
\begin{itemize}
\item $i_0 = j_0 (2m+1)$, and 
\item length $2k = (2r+1)2m+2r$, so $2k + 1= (2r+1)(2m+1)$.
\end{itemize} 
\item There exist $j_0 \not= j_1$ such that $|c_{j_0}| = |c_{j_1}| =2$, and $|c_j|=1$ for $i \not= j_0, j_1$.  In this case $K$ is of type $4b$ in Lemma \ref{braid4} with 
\begin{itemize}
\item $i_0 = j_0 (2m+1)$, 
\item $i_1 = j_1 (2m+1)$, and 
\item length $2k = (2r+1)2m+2r$, so $2k + 1= (2r+1)(2m+1)$.
\end{itemize}
\end{enumerate}

\underline{Case 2}:  $\sum_{\substack{j=1 \\ c_j \not= 0}}^{2r} (|c_j|  - 1)=1$, $2r - \#\{c_j \not= 0\}=1$ and $(2r+1) t(\mathbf{a}) + 2 \#\{c_j \not= 0\} - t(\tilde{\mathbf{a}}) =0$. 
Then there exist $j_0 \not= j_1$ such that $c_{j_0}=0$ (note that $\ve_{j_0+1}=\ve_{j_0}$), $|c_{j_1}| = 2$, and $|c_j|=1$ for $j \not= j_0, j_1$. Moreover $\ve_{j+1} =\ve_{j}$ and $(-1)^j \ve_{j} c_{j} <0$ for $j \not= j_0$. Hence $\ve_1 = \cdots = \ve_{2r+1} = 1$, and $(-1)^{j-1} c_{j} > 0$ for $j \not= j_0$. In this case, by deleting $c_{j_0}=0$ and combining $\ve_{j_0} \mathbf{a}^{(-1)^{j_0-1}}$ with $\ve_{j_0+1} \mathbf{a}^{(-1)^{j_0}} = \ve_{j_0} \mathbf{a}^{(-1)^{j_0}}$, we see that  $K$ is of type $4b$ in Lemma \ref{braid4} with 
\begin{itemize}
\item $i_0 = j_0 (2m+1)-1$, 
\item $i_1= j_1(2m+1)$ if $j_1 < j_0$ and $i_1= j_1 (2m+1) - 2$ if $j_1 > j_0$, and 
\item length $2k = (2r+1)2m+2r-2$, so $2k + 3= (2r+1)(2m+1)$.
\end{itemize}

\underline{Case 3}: $\sum_{\substack{j=1 \\ c_j \not= 0}}^{2r} (|c_j|  - 1)=1$, $2r - \#\{c_j \not= 0\}=0$ and $(2r+1) t(\mathbf{a}) + 2 \#\{c_j \not= 0\} - t(\tilde{\mathbf{a}}) =1$. 
Then there exists $j_0$ such that $|c_{j_0}| =2$ and $|c_j|=1$ for  $j \not= j_0$, and there exists $j_1$ (it might happen that $j_1 = j_0$) such that $\ve_{j_1+1} = - \ve_{j_1}$, $\ve_{j+1} =\ve_{j}$ and $(-1)^j \ve_{j} c_{j} <0$ for $j \not=  j_1$.  This implies that $\ve_j =1$ for $j \le j_1$, and $\ve_j = -1$ for $j \ge j_1+1$. Hence $(-1)^{j-1} c_j >0$ for $j < j_1$, and $(-1)^{j} c_j>0$ for $j > j_1$. In this case $K$ is of type $4c$ in Lemma \ref{braid4} with 
\begin{itemize}
\item $i_0 = j_0 (2m+1)$,  
\item $i_1 = j_1 (2m+1)$ or $i_1 = j_1 (2m+1) -1$ (depending on whether $(-1)^{j_1 - 1} c_{j_1} >0$ or $(-1)^{j_1} c_{j_1} >0$), and 
\item length $2k = (2r+1)2m+2r$, so $2k + 1= (2r+1)(2m+1)$.
\end{itemize}

\underline{Case 4}: $\sum_{\substack{j=1 \\ c_j \not= 0}}^{2r} (|c_j|  - 1)=0$, $2r - \#\{c_j \not= 0\}=2$ and $(2r+1) t(\mathbf{a}) + 2 \#\{c_j \not= 0\} - t(\tilde{\mathbf{a}}) =0$. Then  there exist $j_0<j_1$ such that $c_{j_0} = c_{j_1} = 0$ (note that $\ve_{j_0+1}=\ve_{j_0}$ and $\ve_{j_1+1}=\ve_{j_1}$), and $|c_j|=1$ for $j \not= j_0, j_1$. Moreover $\ve_{j+1} = \ve_j$ and $(-1)^j \ve_{j} c_{j} <0$ for $j \not= j_0, j_1$. This implies that $\ve_1 = \cdots = \ve_{2r+1} = 1$, and $(-1)^{j-1}c_{j} > 0$ for $j \not= j_0, j_1$. In this case, by deleting $c_j=0$ and combining $\ve_j \mathbf{a}^{(-1)^{j-1}}$ with $\ve_{j+1} \mathbf{a}^{(-1)^j} = \ve_{j} \mathbf{a}^{(-1)^j}$ for $j=j_0, j_1$, we see that  $K$ is of type $4b$ in Lemma \ref{braid4} with 
\begin{itemize}
\item $i_0 = j_0 (2m+1)-1$, 
\item $i_1 = j_1 (2m+1)-3$, and 
\item length $2k = (2r+1)2m+2r-4$, so $2k + 5= (2r+1)(2m+1)$.
\end{itemize}

\underline{Case 5}: $\sum_{\substack{j=1 \\ c_j \not= 0}}^{2r} (|c_j|  - 1)=0$, $2r - \#\{c_j \not= 0\}=1$ and $(2r+1) t(\mathbf{a}) + 2 \#\{c_j \not= 0\} - t(\tilde{\mathbf{a}}) =1$. 
Then there exists $j_0$ such that $c_{j_0}=0$ (note that $\ve_{j_0}=\ve_{j_0+1}$) and $|c_j|=1$ for  $j \not= j_0$. There also exists $j_1 \not=  j_0$ such that $\ve_{j_1+1} = - \ve_{j_1}$, $\ve_{j+1} =\ve_{j}$ and $(-1)^j \ve_{j} c_{j} <0$ for $j \not=  j_0, j_1$.  This implies that $\ve_j =1$ for $j \le j_1$, and $\ve_j = -1$ for $j \ge j_1+1$. Hence $(-1)^{j-1} c_{j} >0$ for $j <  j_1$ and $j \not= j_0$, and $(-1)^{j}c_{j} > 0$ for $j > j_1$ and $j \not= j_0$. In this case, by deleting $c_{j_0}=0$ and combining $\ve_{j_0} \mathbf{a}^{(-1)^{j_0-1}}$ with $\ve_{j_0+1} \mathbf{a}^{(-1)^{j_0}} = \ve_{j_0} \mathbf{a}^{(-1)^{j_0}}$, we see that $K$ is of type $4c$ in Lemma \ref{braid4} with 
\begin{itemize}
\item $i_0 = j_0 (2m+1)-1$, 
\item $i_1 = j_1 (2m+1)$ or  $i_1 = j_1 (2m+1) -1$ if $j_1 < j_0$, and $i_1 = j_1 (2m+1)-2$ or $i_1 = j_1 (2m+1)-3$ if $j_1 >  j_0$ (depending on whether $(-1)^{j_1 - 1} c_{j_1} >0$ or $(-1)^{j_1} c_{j_1} >0$), 
\item length $2k = (2r+1)2m+2r-2$, so $2k + 3= (2r+1)(2m+1)$.
\end{itemize}

\underline{Case 6}: $\sum_{\substack{j=1 \\ c_j \not= 0}}^{2r} (|c_j|  - 1)=0$, $2r - \#\{c_j \not= 0\}=0$ and $(2r+1) t(\mathbf{a}) + 2 \#\{c_j \not= 0\} - t(\tilde{\mathbf{a}}) =2$. Then $|c_j|=1$ for all $j$. There are 2 subcases to consider:
\begin{enumerate}
\item There exists $j_0$ such that $\ve_{j_0+1} = \ve_{j_0}$ and $(-1)^{j_0} \ve_{j_0} c_{j_0} >0$, and $\ve_{j+1} =\ve_{j}$ and $(-1)^j \ve_{j} c_{j} <0$ for $j \not= j_0$. Hence $\ve_1 = \cdots = \ve_{2r+1} = 1$, $(-1)^{j_0} c_{j_0}>0$, and $(-1)^{j-1}c_{j}>0$ for $j \not= j_0$. In this case $K$ is of type $4d$ in Lemma \ref{braid4} with
\begin{itemize}
\item $i_0 = (2m+1)j_0-1$, 
\item $i_1 = (2m+1)j_0$, 
\item length $2k = (2r+1)2m+2r$, so $2k + 1= (2r+1)(2m+1)$.
 \end{itemize}
\item There exist $j_0 < j_1$ such that $\ve_{j+1} = - \ve_{j}$ for $j = j_0, j_1$, and $\ve_{j+1} =\ve_{j}$ and $(-1)^j \ve_{j} c_{j} <0$ for $j \not= j_0, j_1$. This implies that $\ve_j =1$ for $j \le j_0$ or $j \ge j_1 +1$, and $\ve_j =-1$ for $j_0 +1 \le j \le j_1$. Hence $(-1)^{j-1}c_j > 0$ for $j < j_0$ or $j > j_1$, and $(-1)^j c_j > 0$ for $j_0 < j < j_1$. In this case $K$ is of type $4d$ in Lemma \ref{braid4} with 
\begin{itemize}
\item $i_0 = j_0 (2m+1)$ or $i_0 = j_0 (2m+1)-1$  (depending on whether $(-1)^{j_0 - 1} c_{j_0} >0$ or $(-1)^{j_0} c_{j_0} >0$), 
\item $i_1 = j_1 (2m+1)-1$ or $i_1 = j_1 (2m+1)$ (depending on whether $(-1)^{j_1 - 1} c_{j_1} >0$ or $(-1)^{j_1} c_{j_1} >0$), 
\item length $2k = (2r+1)2m+2r$, so $2k + 1= (2r+1)(2m+1)$.
\end{itemize}
\end{enumerate}

Hence we obtain the following proposition. 

\begin{proposition} \label{braid4minimal}
Let $K$ be a two-bridge knot. If $\emph{braid}(K)=4$, then $K$ is not minimal if and only if one of the following holds
\begin{itemize}
\item[$-$] (type $4A$) $K=K([s_1 n_1, s_2 n_2, \cdots, s_{2k} n_{2k}])$ where $s_i = (-1)^{i-1}$, 
and there exists $(r,m,i_0,j_0) \in {\mathbb N}^4$ such that 
$2k+1 = (2r+1)(2m+1)$, $i_0 = j_0 (2m+1) ~(1 \leq j_0 \leq 2 r)$, $n_{i_0} = 6$ and $n_i=2$ for $i \not= i_0$.
\item[$-$]  (type $4B1$) $K=K([s_1 n_1, s_2 n_2, \cdots, s_{2k} n_{2k}])$ where $s_i = (-1)^{i-1}$, 
and there exists $(r,m,i_0,i_1,j_0,j_1) \in {\mathbb N}^6$ such that 
$2k+1 = (2r+1)(2m+1)$, $i_0 = j_0 (2m+1), i_1 = j_1 (2m+1) ~(1 \leq j_0 \not= j_1 \leq 2 r)$, $n_{i_0} = n_{i_1} = 4$ and $n_i=2$ for $i \not= i_0,i_1$.
\item[$-$]  (type $4B2$) $K=K([s_1 n_1, s_2 n_2, \cdots, s_{2k} n_{2k}])$ where $s_i = (-1)^{i-1}$, 
and there exists $(r,m,i_0,i_1,j_0,j_1) \in {\mathbb N}^6$ such that 
$2k+3 = (2r+1)(2m+1)$, $i_0 = j_0 (2m+1)-1$,  $i_1= j_1 (2m+1)$ if $j_1 < j_0$ and $i_1= j_1(2m+1) - 2$ if $j_1 > j_0 ~(1 \leq j_0 \not= j_1 \leq 2 r)$, $n_{i_0} = n_{i_1} = 4$ and $n_i=2$ for $i \not= i_0,i_1$.
\item[$-$]  (type $4B3$) $K=K([s_1 n_1, s_2 n_2, \cdots, s_{2k} n_{2k}])$ where $s_i = (-1)^{i-1}$, 
and there exists $(r,m,i_0,i_1,j_0,j_1) \in {\mathbb N}^6$ such that 
$2k+5 = (2r+1)(2m+1)$, $i_0 = j_0 (2m+1)-1$, $i_1 = j_1 (2m+1)-3 ~(1 \leq j_0 < j_1 \leq 2 r)$, $n_{i_0} = n_{i_1} = 4$ and $n_i=2$ for $i \not= i_0,i_1$.

\item[$-$]  (type $4C1$) $K=K([s_1 n_1, s_2 n_2, \cdots, s_{2k} n_{2k}])$ where 
there exist $(r,m,i_0,i_1,j_0,j_1) \in {\mathbb N}^6$ such that $2k+1 = (2r+1)(2m+1)$, $i_0 = j_0 (2m+1)$, $i_1 = j_1 (2m+1)$ or  $i_1 = j_1 (2m+1) - 1~(1 \leq j_0, j_1 \leq 2 r)$, 
$s_i = (-1)^{i-1}$ for $i \le i_1$ and $s_i = (-1)^{i}$ for $i \ge i_1+1$, $n_{i_0}=4$ and $n_i=2$ for $i \not= i_0$. 

\item[$-$]  (type $4C2$) $K=K([s_1 n_1, s_2 n_2, \cdots, s_{2k} n_{2k}])$ where 
there exist $(r,m,i_0,i_1,j_0,j_1) \in {\mathbb N}^6$ such that $2k+3 = (2r+1)(2m+1)$, $i_0 = j_0 (2m+1)-1$, 
 $i_1 = j_1 (2m+1)$ or  $i_1 = j_1 (2m+1)-1$ if $j_1 < j_0$, and $i_1 = j_1 (2m+1)-2$ or $i_1 = j_1 (2m+1)-3$ if $j_1 > j_0$,  $(1 \leq j_0 \not= j_1 \leq 2 r)$, 
$s_i = (-1)^{i-1}$ for $i \le i_1$ and $s_i = (-1)^{i}$ for $i \ge i_1+1$, $n_{i_0}=4$ and $n_i=2$ for $i \not= i_0$.

\item[$-$]  (type $4D$) $K=K([2s_1, 2s_2, \cdots, 2s_{2k}])$ where 
there exist $(r,m,i_0,i_1,j_0,j_1) \in {\mathbb N}^6$ such that $2k + 1= (2r+1)(2m+1)$, $i_0 = j_0 (2m+1)-1$ or $i_0 = j_0 (2m+1)$, $i_1 = j_1 (2m+1)-1$ or $i_1 = j_1 (2m+1) ~(1 \leq j_0 \le   j_1 \leq 2 r)$ and $i_0 < i_1$, $s_i = (-1)^{i-1}$ for $i \leq i_0$ or $i \ge i_1+1$ and $s_i = (-1)^{i}$ for $i_0+1 \le i \le i_1$.

\end{itemize}
\end{proposition}

\subsection{Examples}

In this subsection, we provide all the non minimal two-bridge knots with braid index up to $4$ and up to $15$ crossings. 
Then all the knots with braid index up to $4$ and up to $15$ crossings which are not on the Table \ref{table_nonminimal} are minimal. 
Here we do not distinguish a knot with its mirror image. 

\begin{table}[ht]
\begin{tabular}{c|c|c|l|l}
braid index & type & $c(K)$ & even continued fraction & onto  \\
\hline
$2$ & $2$ & $9$ & $[2, -2, 2, -2, 2, -2, 2, -2]$ & $3_1$ \\ 
$2$ & $2$ & $15$ & $[2, -2, 2, -2, 2, -2, 2, -2, 2, -2, 2, -2, 2, -2]$ & $3_1$ and $5_1$ \\
\hline 
$3$ & $3A1$ & $11$ & $[2, -2, 2, -2, 2, -4, 2, -2]$ & $3_1$ \\ 
$3$ & $3A2$ & $9$ & $[2, -4, 2, -2, 2, -2]$ & $3_1$ \\ 
$3$ & $3A2$ & $15$ & $[2, -4, 2, -2, 2, -2, 2, -2, 2, -2, 2, -2]$ & $3_1$ \\ 
$3$ & $3A2$ & $15$ & $[2, -2, 2, -2, 2, -2, 2, -4, 2, -2, 2, -2]$ & $3_1$ \\ 
$3$ & $3A2$ & $15$ & $[2, -2, 2, -4, 2, -2, 2, -2, 2, -2, 2, -2]$ & $5_1$ \\ 
$3$ & $3B$ & $10$ & $[2, -2, -2, 2, -2, 2, -2, 2]$ & $3_1$ \\ 
$3$ & $3B$ & $10$ & $[2, -2, 2, -2, 2, 2, -2, 2]$ & $3_1$ \\ 
\hline
$4$ & $4A$ & $13$ & $[2, -2, 2, -2, 2, -6, 2, -2]$ & $3_1$ \\ 
$4$ & $4B1$ & $13$ & $[2, -2, 4, -2, 2, -4, 2, -2]$ & $3_1$ \\ 
$4$ & $4B2$ & $11$ & $[2, -4, 2, -4, 2, -2]$ & $3_1$ \\ 
$4$ & $4B3$ & $9$ & $[2, -4, 4, -2]$ & $3_1$ \\ 
$4$ & $4B3$ & $15$ & $[2, -4, 2, -2, 2, -4, 2, -2, 2, -2]$ & $3_1$ \\ 
$4$ & $4B3$ & $15$ & $[2, -4, 2, -2, 2, -2, 2, -2, 4, -2]$ & $3_1$ \\ 
$4$ & $4B3$ & $15$ & $[2, -4, 4, -2, 2, -2, 2, -2, 2, -2]$ & $3_1$ \\ 
$4$ & $4B3$ & $15$ & $[2, -2, 2, -2, 4, -4, 2, -2, 2, -2]$ & $3_1$ \\ 
$4$ & $4B3$ & $15$ & $[2, -2, 2, -4, 2, -2, 4, -2, 2, -2]$ & $5_1$ \\ 
$4$ & $4C1$ & $12$ & $[2, -2, -4, 2, -2, 2, -2, 2]$ & $3_1$ \\ 
$4$ & $4C1$ & $12$ & $[2, -2, -2, 2, -2, 4, -2, 2]$ & $3_1$ \\ 
$4$ & $4C1$ & $12$ & $[2, -2, 2, -2, 2, 4, -2, 2]$ & $3_1$ \\ 
$4$ & $4C1$ & $12$ & $[2, -2, 2, 2, -2, 4, -2, 2]$ & $3_1$ \\ 
$4$ & $4C2$ & $10$ & $[2, -4, 2, -2, -2, 2]$ & $3_1$ \\ 
$4$ & $4C2$ & $10$ & $[2, -4, 2, 2, -2, 2]$ & $3_1$ \\ 
$4$ & $4D$ & $11$ & $[2, -2, -2, -2, 2, -2, 2, -2]$ & $3_1$ \\ 
$4$ & $4D$ & $11$ & $[2, -2, -2, 2, -2, -2, 2, -2]$ & $3_1$ \\ 
$4$ & $4D$ & $11$ & $[2, -2, -2, 2, -2, 2, 2, -2]$ & $3_1$ \\ 
$4$ & $4D$ & $11$ & $[2, -2, 2, 2, -2, -2, 2, -2]$ & $3_1$  
\end{tabular}
\bigskip
\caption{non minimal two-bridge knots with braid index up to $4$ and up to $15$ crossings}
\label{table_nonminimal}
\end{table}

\section{Average braid index} \label{avg}

In this section, we compute the average braid index of all the two-bridge knots with a given crossing number. 

We first recall a  known result about the number of two-bridge knots with $c$ crossings. 

\begin{theorem}[Ernst-Sumners \cite{ernstsumners}] \label{thmernstsumners}
For $c \geq 3$, the number of two-bridge knots with $c$ crossings is given by 
\[
TK(c) = 
\left\{
\begin{array}{ll}
\displaystyle{\frac{1}{3} (2^{c-2} - 1)} & c \equiv 0 \pmod{2} 
\smallskip\\
\displaystyle{\frac{1}{3} (2^{c-2} + 2^{(c-1)/2})} & c \equiv 1 \pmod{4} 
\smallskip\\ 
\displaystyle{\frac{1}{3} (2^{c-2} + 2^{(c-1)/2} + 2)} & c \equiv 3 \pmod{4} 
\end{array}
\right. .
\]
\end{theorem}

Let $\overline{\mathrm{braid}}_c$ denote the average braid index of all the two-bridge knots with $c$ crossings. We will prove the following. 

\begin{theorem}\label{mainthm}
For $c \geq 3$, we have
\[
\overline{\emph{braid}}_c =
\left\{
\begin{array}{ll}
\displaystyle{\frac{3c+11}{9}  + \frac{2c-4}{3(2^{c-2}-1)}} & c \equiv 0 \pmod{2} 
\smallskip \\
\displaystyle{\frac{3c+11}{9}   - \frac{2^{(c+3)/2} + 9 c - 19}{9 (2^{c-2} + 2^{(c-1)/2})}} & c \equiv 1 \pmod{4} 
\smallskip\\
\displaystyle{\frac{3c+11}{9}   - \frac{2^{(c+3)/2} + 3 c - 5}{9 (2^{c-2} + 2^{(c-1)/2}+2)}} & c \equiv 3 \pmod{4}
\end{array}
\right. .
\]
In particular,  $\overline{\emph{braid}}_c \sim \displaystyle{\frac{c}{3} + \frac{11}{9}}$ as $c \to \infty$.
\end{theorem}

Recall  that  the crossing number and braid index of a two-bridge knot $K=K([2 a_1,2 a_2, \ldots, 2 a_{2m}])$, where $a_i \neq 0$,  are given by
\begin{eqnarray*}
c(K) &=& \left(\sum_{i=1}^{2m} 2|a_i| \right) - \ell, \\
\text{braid}(K) &=& \left(\sum_{i=1}^{2m} |a_i| \right) - \ell +1,
\end{eqnarray*}
where $\ell = t(\mathbf{a})$ is the number of sign changes in the sequence $\mathbf{a} = (2 a_1, 2 a_2, \ldots, 2 a_{2m})$. This implies that $\mathrm{braid}(K) = \frac{1}{2} c(K) +1 - \frac{1}{2}  \ell$. Note that $0 \le \ell \le 2m-1$, $\ell$ and $c$ have the same parity, and $(c + \ell)/2=\sum_{i=1}^{2m} |a_i| \ge 2m$. Hence $\lceil (\ell+1)/2 \rceil \le m \le \lfloor (c + \ell)/4 \rfloor$. In particular, $\ell \le c-4$ if $c$ is even and $\ell \le c-2$ if $c$ is odd.

Let $N_{c,\ell}$ denote the number of two-bridge knots $K([2a_1, 2a_2, \dots, 2a_{2m}])$  with crossing number $c$ and with number of sign changes $\ell$ in the sequence $(2a_1, 2a_2, \dots, 2a_{2m})$. Then we have the following. 

\begin{proposition} \label{A1}
If $(c + \ell)/2$ is even and $\ell$ is odd, then
\begin{eqnarray*}
N_{c,\ell} 
&=& \binom{(c + \ell)/2-1}{\ell} \sum_{(\ell+1)/2 \le m  \le (c + \ell)/4}    \binom{(c - \ell)/2-1}{2m-1-\ell} \\
&&  + \,  \binom{(c + \ell)/4-1}{(\ell-1)/2} \sum_{(\ell+1)/2 \le m  \le (c + \ell)/4}     \binom{(c - \ell)/4-1/2}{m-1-(\ell-1)/2}.
\end{eqnarray*} 
Otherwise,
$$
N_{c,\ell} = \binom{(c + \ell)/2-1}{\ell}  \sum_{(\ell+1)/2 \le  m \le (c + \ell)/4}  \binom{(c - \ell)/2-1}{2m-1-\ell}.
$$
\end{proposition}

We obtained the total genus $A_{c,\ell}$ in \cite[Proposition 3.2]{ST2}, which is the sum of genera of all two-bridge knots $K([2 a_1,2 a_2, \ldots, 2 a_{2m}])$ with crossing number $c$ and with number of sign changes $\ell$ in the sequence $(2a_1, 2a_2, \dots, 2a_{2m})$. 
By forgetting $m$ in each term of $A_{c,\ell}$, we get Proposition \ref{A1}. 

Similarly, we use the following identities which are shown in \cite[Lemmas 2.2 and 2.3]{ST2}. 

\begin{lemma}\label{lem22} 
We have
\begin{eqnarray*}
\sum_{q=0}^{n-1}  2^{q} \binom{2n-1-q}{q} &=& \frac{4^n-1}{3}, \\
\sum_{q=0}^{n}  2^{q} \binom{2n-q}{q} &=& \frac{2 \cdot 4^n + 1}{3}.
\end{eqnarray*}
\end{lemma}

\begin{lemma}\label{lem23}
We have 
\begin{eqnarray*}
\sum_{q=0}^{n-1}  q \, 2^{q} \binom{2n-1-q}{q}  &=& \frac{2}{27} \left(  (3 n- 2) 4^n - 6n + 2 \right), \\
\sum_{q=0}^{n}  q \, 2^{q} \binom{2n-q}{q}  &=& \frac{2}{27} \left(  (6n -1) 4^n + 6n + 1 \right).
\end{eqnarray*}
\end{lemma}

We are now ready to prove Theorem \ref{mainthm}. The proof is divided into 2 cases, depending on the parity of the crossing number $c$. 

\subsection{Even crossing number}

We first consider the case $c=2k$, $k \in \BZ$. 
Since $c$ and $\ell$ have the same parity, we may write $\ell = 2l$, $l \in \BZ$. Since $\lceil (\ell+1)/2 \rceil \le m \le \lfloor (c + \ell)/4 \rfloor$, we have 
$l+1 \le m \le \lfloor (k+l)/2 \rfloor$. In particular, $0 \le l \le k-2$. 

Since $\ell$ is even, by Proposition \ref{A1} we have
$$
N_{2k,2l} =  \binom{k+l-1}{2l} \sum_{m=l+1}^{\lfloor (k+l)/2 \rfloor}  \binom{k-l-1}{2m-2l-1}.
$$

We have the following lemma in \cite[$Q$ in Proof of Lemma 3.3]{ST2}.

\begin{lemma} \label{A-simplify}
We have 
$$
\sum_{m=l+1}^{\lfloor (k+l)/2 \rfloor}   \binom{k-l-1}{2m-2l-1} = 
2^{k-l-2}.
$$
\end{lemma}


The total ``sign change" of two-bridge knots with crossing number $c=2k$ is equal to 
\begin{eqnarray*}
TS(2k) &=&  \sum_{l=0}^{k-2}  2l N_{2k,2l} \\
&=&  \sum_{l=0}^{k-2} 2l \binom{k+l-1}{2l} 2^{k-l-2}\\
&=&   -(k-1) + \sum_{l=0}^{k-1} l \binom{k+l-1}{k-l-1} 2^{k-l-1} \\
 &=&  -(k-1) + (k-1) \sum_{l=0}^{k-1}  2^{k-l-1} \binom{k+l-1}{k-l-1}  - \sum_{l=0}^{k-1}  (k-l-1) 2^{k-l-1} \binom{k+l-1}{k-l-1} \\
 &=&  -(k-1) + (k-1) \sum_{q=0}^{k-1}  2^{q} \binom{2k-2-q}{q}  -  \sum_{q=0}^{k-1}  q \, 2^{q} \binom{2k-2-q}{q} \\
 &=&  -(k-1) + (k-1) \frac{2 \cdot 4^{k-1} +1}{3} -  \frac{2}{27} \left(  (6k-7) 4^{k-1} + 6k-5 \right) \\
&=& \frac{2}{27}  \left( (3k-2) 4^{k-1} - 15k +14 \right)\\
&=& \frac{1}{27} \left( (3c-4) 2^{c-2} - 15c+28\right), \qquad c= 2k.
\end{eqnarray*}

Therefore if $c \equiv 0 \pmod{2}$, then the average ``sign change" of two-bridge knots with crossing number $c$ is equal to 
$$
\overline{\ell}_{c} =  \frac{TS(c)} {TK(c)} = \frac{(3c-4) 2^{c-2} - 15c+28}{9(2^{c-2}-1)} = \frac{3c-4}{9}  - \frac{4c-8}{3(2^{c-2}-1)}.
$$
Note that $\mathrm{braid}(K) = \frac{1}{2} c(K) +1 - \frac{1}{2}  \ell$. Hence 
$$
\overline{\text{braid}}_{c} = \frac{1}{2} c +1 - \frac{1}{2} \overline{\ell}_{c} =\frac{3c+11}{9}  + \frac{2c-4}{3(2^{c-2}-1)}.
$$

\subsection{Odd crossing number} 
We now consider the case $c=2k+1$, $k \in \BZ$. Since $c$ and $\ell$ have the same parity, we may write $\ell = 2l+1$, $l \in \BZ$. Since $\lceil (\ell+1)/2 \rceil \le m \le \lfloor (c + \ell)/4 \rfloor$, we have 
$l+1 \le m \le \lfloor (k+l+1)/2 \rfloor$. In particular, $0 \le l \le k-1$.

Since $\ell$ is odd, by Proposition \ref{A1} we have 
\begin{eqnarray*}
N_{2k+1,2l+1} 
&=& \binom{k+l}{2l+1}  \sum_{m=l+1}^{\lfloor (k+l+1)/2 \rfloor} \binom{k-l-1}{2m-2l-2} \\
&&  + \,    \frac{1+(-1)^{k+l+1}}{2} \binom{(k+l-1)/2}{l} \sum_{m=l+1}^{\lfloor (k+l+1)/2 \rfloor}    \binom{(k-l-1)/2}{m-l-1}.
\end{eqnarray*} 

\begin{lemma} \label{A-simplify1}
We have 
$$\sum_{m=l+1}^{\lfloor (k+l+1)/2 \rfloor}   \binom{k-l-1}{2m-2l-2} = 
\begin{cases}
  2^{k-l-2}  & \text{ if } ~0 \le l \le k-2, \\
  2^{k-l-2}  + 1/2 & \text{ if } ~l = k-1.
\end{cases}
$$
\end{lemma}

\begin{proof}
If $0 \le l \le k-2$ then  
\begin{eqnarray*}
\sum_{m=l+1}^{\lfloor (k+l+1)/2 \rfloor} \binom{k-l-1}{2m-2l-2}&=& \sum_{m'=0}^{\lfloor (k-l-1)/2 \rfloor} \binom{k-l-1}{2m'} \\
&=&  \frac{1}{2} \left[ \sum_{m''=0}^{k-l-1} \binom{k-l-1}{m''} + \sum_{m''=0}^{k-l-1} (-1)^{m''}\binom{k-l-1}{m''} \right] \\
&=& \frac{1}{2}  \left( 2^{k-l-1} + 0^{k-l-1} \right) = 2^{k-l-2}.
\end{eqnarray*}
If $l=k-1$ then $\sum_{m=l+1}^{\lfloor (k+l+1)/2 \rfloor} \binom{k-l-1}{2m-2l-2}=1=2^{k-l-2}  + 1/2$. 
\end{proof}

\begin{lemma} \label{A-simplify2}
If $k+l+1$ is even, then 
$$
\sum_{m=l+1}^{\lfloor (k+l+1)/2 \rfloor}    \binom{(k-l-1)/2}{m-l-1} = 2^{(k-l-1)/2}.
$$
\end{lemma}

\begin{proof}
Since $k+l+1$ is even, $k-l-1$ is also even and 
\[
\sum_{m=l+1}^{\lfloor (k+l+1)/2 \rfloor}    \binom{(k-l-1)/2}{m-l-1} 
=  \sum_{m'=0}^{(k-l-1)/2 } \binom{(k-l-1)/2}{m'} = 2^{(k-l-1)/2}.
\]
\end{proof}

By using these lemmas, the total ``sign change" of two-bridge knots with crossing number $c=2k+1$ is equal to
\begin{eqnarray*}
TS(2k+1) &=& \sum_{l=0}^{k-1} (2l+1) N_{2k+1, 2l+1}\\
&=&  \frac{2l+1}{2} \binom{k+l}{2l+1} \Big|_{l=k-1}  + \sum_{l=0}^{k-1}  (2l+1)  2^{k-l-2} \binom{k+l}{2l+1}  \\
&& + \sum_{\substack{l=0 \\ k + l +1 \text{ is even}}}^{k-1}  (2l+1)  2^{(k-l-1)/2} \binom{(k+l-1)/2}{l}\\
&=&  \frac{2k-1}{2}   +\frac{1}{2} \sum_{l=0}^{k-1}  (2l+1)  2^{k-l-1} \binom{k+l}{2l+1}  \\
&& + \sum_{\substack{l=0 \\ k + l +1 \text{ is even}}}^{k-1}  (2l+1)  2^{(k-l-1)/2} \binom{(k+l-1)/2}{l}.
\end{eqnarray*}

The first sum in the above expression can be simplified as follows
\begin{eqnarray*}
&& \sum_{l=0}^{k-1}  (2l+1)  2^{k-l-1} \binom{k+l}{2l+1}\\
 &=&  \sum_{l=0}^{k-1} (2l+1)  2^{k-l-1}\binom{k+l}{k-l-1} \\
 &=&  (2k-1) \sum_{l=0}^{k-1}  2^{k-l-1} \binom{k+l}{k-l-1} - 2\sum_{l=0}^{k-1}  (k-l-1) 2^{k-l-1} \binom{k+l}{k-l-1} \\
 &=& (2k -1) \sum_{q=0}^{k-1}  2^{q} \binom{2k-1-q}{q} - 2\sum_{q=0}^{k-1}  q \, 2^{q} \binom{2k-1-q}{q} \\
 &=& \frac{(2k-1) (4^k-1)}{3} -  \frac{4}{27} \left(  (3 k- 2) 4^k - 6k + 2 \right) \\
 &=& \frac{1}{27} \left( (6k-1) 4^k + 6k+1 \right) \\
 &=& \frac{1}{27} \left( (3c-4) 2^{c-1} + 3c-2 \right), \qquad c= 2k+1.
\end{eqnarray*}

The second sum  depends on the parity of $k$. 
If $k=2n$, then 
\begin{eqnarray*}
&& \sum_{\substack{l=0 \\ k + l +1 \text{ is even}}}^{k-1}  (2l+1)  2^{(k-l-1)/2} \binom{(k+l-1)/2}{l}   \\
&=& \sum_{p=0}^{n-1}(4p+3) 2^{n-p-1}  \binom{n+p}{2p+1} \\
&=& (4n-1)  \sum_{p=0}^{n-1}  2^{n-p-1} \binom{n+p}{n-p-1}  -  4\sum_{p=0}^{n-1} (n-p-1)\, 2^{n-p-1} \binom{n+p}{n-p-1}  \\
&=& (4n-1)   \sum_{q=0}^{n-1} 2^q \binom{2n-1-q}{q} - 4\sum_{q=0}^{n-1} q \, 2^{q} \binom{2n-1-q}{q} \\
&=& \frac{(4n-1) (4^n-1)}{3} -  \frac{8}{27} \left(  (3 n- 2) 4^n - 6n + 2 \right) \\
 &=& \frac{1}{27} \left( (12n+7) 4^n + 12n-7 \right) \\
&=& \frac{1}{27} \left( (3c+4) \cdot 2^{(c-1)/2} +3c-10 \right), \qquad c = 4n+1 . 
\end{eqnarray*} 
Therefore if $c \equiv 1 \pmod{4}$ then
\begin{eqnarray*}
TS(c) 
&=& \frac{c-2}{2}+ \frac{1}{2} \cdot \frac{1}{27} \left( (3c-4) 2^{c-1} + 3c-2 \right) + \frac{1}{27} \left( (3c+4) \cdot 2^{(c-1)/2} +3c-10 \right) \\
&=& \frac{1}{27} \left( (3 c -4) \cdot 2^{c-2} + (3 c + 4)\cdot 2^{(c-1)/2} + 18c - 38 \right).
\end{eqnarray*}
Hence the average ``sign change" of two-bridge knots with crossing number $c\equiv 1 \pmod{4}$ is equal to 
\begin{eqnarray*}
\overline{\ell}_{c} &=&  \frac{TS(c)} {TK(c)} = \frac{ (3 c -4) \cdot 2^{c-2} + (3 c + 4)\cdot 2^{(c-1)/2} + 18c - 38}{9 (2^{c-2} + 2^{(c-1)/2})} \\ 
&=& \frac{3c-4}{9}  +  \frac{8 \cdot 2^{(c-1)/2} + 18 c - 38}{9 (2^{c-2} + 2^{(c-1)/2})}.
\end{eqnarray*}
Then 
$$
\overline{\text{braid}}_{c} = \frac{1}{2} c +1 - \frac{1}{2} \overline{\ell}_{c} =\frac{3c+11}{9}   - \frac{2^{(c+3)/2} + 9 c - 19}{9 (2^{c-2} + 2^{(c-1)/2})}.
$$

If $k=2n+1$, then 
\begin{eqnarray*}
&&\sum_{\substack{l=0 \\ k + l +1 \text{ is even}}}^{k-1}  (2l+1)  2^{(k-l-1)/2} \binom{(k+l-1)/2}{l} \\
&=& \sum_{p=0}^{n}(4p+1) 2^{n-p} \binom{n+p}{2p} \\
&=& (4n+1)  \sum_{p=0}^{n}  2^{n-p} \binom{n+p}{n-p}   - 4 \sum_{p=0}^{n}  (n-p)\, 2^{n-p} \binom{n+p}{n-p} \\
&=& (4n+1)   \sum_{q=0}^{n}  2^{q} \binom{2n-q}{q}  -  4\sum_{q=0}^{n}   q \, 2^{q} \binom{2n-q}{q}\\
&=& (4n+1) \, \frac{2 \cdot 4^n + 1}{3}-  \frac{8}{27} \left(  (6n -1) 4^n + 6n + 1 \right)\\
 &=& \frac{1}{27} \left( (24n+26) 4^n -12n+1 \right) \\
&=& \frac{1}{27} \left( (3c+4) \cdot 2^{(c-1)/2} -3c+10 \right), \qquad c = 4n+3 . 
\end{eqnarray*} 
Therefore if $c \equiv 3 \pmod{4}$ then
\begin{eqnarray*}
TS(c) 
&=& \frac{c-2}{2}+ \frac{1}{2} \cdot \frac{1}{27} \left( (3c-4) 2^{c-1} + 3c-2 \right) + \frac{1}{27} \left( (3c+4) \cdot 2^{(c-1)/2} - 3c+10\right) \\
&=& \frac{1}{27} \left( (3 c -4) \cdot 2^{c-2} + (3 c + 4)\cdot 2^{(c-1)/2} + 12c - 18 \right).
\end{eqnarray*}
Hence the average ``sign change" of two-bridge knots with crossing number $c \equiv 3 \pmod{4}$ is equal to 
\begin{eqnarray*}
\overline{\ell}_{c} &=&  \frac{TS(c)} {TK(c)} = \frac{ (3 c -4) \cdot 2^{c-2} + (3 c + 4)\cdot 2^{(c-1)/2} + 12c - 18}{9 (2^{c-2} + 2^{(c-1)/2} + 2)} \\ 
&=& \frac{3c-4}{9}  +  \frac{8 \cdot 2^{(c-1)/2} + 6 c - 10}{9 (2^{c-2} + 2^{(c-1)/2}+2)}.
\end{eqnarray*}
Then 
$$
\overline{\text{braid}}_{c} = \frac{1}{2} c +1 - \frac{1}{2} \overline{\ell}_{c} =\frac{3c+11}{9}   - \frac{2^{(c+3)/2} + 3 c - 5}{9 (2^{c-2} + 2^{(c-1)/2}+2)}.
$$

\subsection{Two-bridge knots up to mirror}

In this subsection, a knot is regarded as the same as its mirror image. 
Namely, the following four two-bridge knots are considered as the same knot: 
\begin{align*}
&K([2a_1,2a_2, \ldots, 2a_{2m}]), K([-2a_1,-2a_2, \ldots, -2a_{2m}]), \\
&K([2a_{2m},\ldots,2 a_2, 2a_1]), K([-2a_{2m}, \ldots,- a_2, -2a_1]).
\end{align*}

Ernst and Sumners in \cite{ernstsumners} also gave the number of two-bridge knots with respect to crossing number, up to mirror image. 

\begin{theorem}[Ernst-Sumners \cite{ernstsumners}]\label{thmernstsumnersuptomirror}
For $c \geq 3$, the number of two-bridge knots with $c$ crossings up to mirror image is given by 
\[
TK^*(c) = 
\left\{
\begin{array}{ll}
\displaystyle{\frac{1}{3} (2^{c-3} +  2^{(c-4)/2})} & c \equiv 0 \pmod{4} 
\smallskip\\
\displaystyle{\frac{1}{3} (2^{c-3} +  2^{(c-3)/2})} & c \equiv 1 \pmod{4} 
\smallskip\\
\displaystyle{\frac{1}{3} (2^{c-3} + 2^{(c-4)/2} - 1)} & c \equiv 2 \pmod{4} 
\smallskip\\ 
\displaystyle{\frac{1}{3} (2^{c-3} + 2^{(c-3)/2} + 1)} & c \equiv 3 \pmod{4} 
\end{array}
\right. .
\]
\end{theorem}

By a similar argument, we obtain the average braid index of two-bridge knots with respect to crossing number, up to mirror image. 

\begin{theorem}\label{mainthm2}
For $c \geq 3$, we have
\[
\overline{\emph{braid}}_c^* =
\left\{
\begin{array}{ll}
\displaystyle{\frac{3c+11}{9}  + \frac{2^{c/2} + 9 c - 16}{9 (2^{c-2} + 2^{(c-2)/2})}} & c \equiv 0 \pmod{4} 
\smallskip \\
\displaystyle{\frac{3c+11}{9}   - \frac{2^{(c+3)/2} + 9 c - 19}{9 (2^{c-2} + 2^{(c-1)/2})}} & c \equiv 1 \pmod{4} 
\smallskip\\
\displaystyle{\frac{3c+11}{9}  + \frac{2^{c/2} + 3 c - 8}{9 (2^{c-2} + 2^{(c-2)/2}-2)}} & c \equiv 2 \pmod{4} 
\smallskip \\
\displaystyle{\frac{3c+11}{9}   - \frac{2^{(c+3)/2} + 3 c - 5}{9 (2^{c-2} + 2^{(c-1)/2}+2)}} & c \equiv 3 \pmod{4}
\end{array}
\right. .
\]
In particular,  $\overline{\emph{braid}}_c^* \sim \displaystyle{\frac{c}{3} + \frac{11}{9}}$ as $c \to \infty$.
\end{theorem}

\begin{proof}[Sketch of Proof]
For $c \geq 3$, we get the total ``sign change" of two-bridge knots with respect to crossing number up to mirror image, which is denoted by $TS^*(c)$ 
in a similar way:

\[
TS^*(c) =
\left\{
\begin{array}{ll}
\displaystyle{\frac{1}{54} \left( (3 c -4) \cdot 2^{c-2} + (3 c - 8)\cdot 2^{(c-2)/2} - 18c + 32 \right)} & c \equiv 0 \pmod{4} 
\smallskip\\
\displaystyle{\frac{1}{54} \left( (3 c -4) \cdot 2^{c-2} + (3 c + 4)\cdot 2^{(c-1)/2} + 18c - 38 \right)} & c \equiv 1 \pmod{4}
\smallskip \\
\displaystyle{\frac{1}{54} \left( (3 c -4) \cdot 2^{c-2} + (3 c - 8)\cdot 2^{(c-2)/2} - 12c + 24 \right)} & c \equiv 2 \pmod{4}
\smallskip \\ 
\displaystyle{\frac{1}{54} \left( (3 c -4) \cdot 2^{c-2} + (3 c + 4)\cdot 2^{(c-1)/2} + 12c - 18 \right)} & c \equiv 3 \pmod{4} 
\end{array}
\right. .
\]

By the equality $\overline{\text{braid}}_{c}^* = \frac{1}{2} c +1 - \frac{1}{2} \cdot \frac{TS^*(c)}{TK^*(c)}$, 
we obtain the desired statement. 
\end{proof}

As a final example, we provide Table \ref{table_averagebraid} which shows the values of $TK(c)$, $TS(c)$, and $\overline{{\rm braid}}_c$. 

\renewcommand{\arraystretch}{1.3}
\begin{table}[ht]
\begin{tabular}{c||c|c|c|c|c|c|c|c|c|c|c|c|c}
$c$ & 3 & 4 & 5 & 6 & 7 & 8 & 9 & 10 & 11 & 12 & 13 & 14 & 15  \\
\hline
\hline
$TK(c)$ & 2 & 1 & 4 & 5 & 14 & 21 & 48 & 85 & 182 & 341 & 704 & 1365 & 2774  \\
\hline
$TS(c)$ & 2 & 0 & 8 & 6 & 30 & 44 & 132 & 242 & 598 & 1208 & 2764 & 5758 & 12678  \\
\hline
$\overline{{\rm braid}}_c$ & $2$ & $3$ & $\frac{5}{2}$ & $\frac{17}{5}$ & $\frac{24}{7}$ & $\frac{83}{21}$ 
& $\frac{33}{8}$ & $\frac{389}{85}$ & $\frac{34}{7}$ & $\frac{1783}{341}$ & $\frac{1949}{352}$ & $\frac{8041}{1365}$ & $\frac{8620}{1387}$  \\
\hline
\hline
$TK^*(c)$ & 1 & 1 & 2 & 3 & 7 & 12 & 24 & 45 & 91 & 176 & 352 & 693 & 1387 \\
\hline
$TS^*(c)$ & 1 & 0 & 4 & 4 & 15 & 24 & 66 & 128 & 299 & 620 & 1382 & 2920 & 6339 \\
\hline
$\overline{{\rm braid}}_c^*$ & 2 & 3 & $\frac{5}{2}$ & $\frac{10}{3}$ & $\frac{24}{7}$ & $4$ 
& $\frac{33}{8}$ & $\frac{206}{45}$ & $\frac{34}{7}$ & $\frac{461}{88}$ & $\frac{1949}{352}$ & $\frac{4084}{693}$ & $\frac{8620}{1387}$ \\
\end{tabular}
\bigskip
\caption{$\overline{{\rm braid}}_c$ and $\overline{{\rm braid}}_c^*$ }
\label{table_averagebraid}
\end{table}
\renewcommand{\arraystretch}{1}

\section*{Acknowledgements}

The first author is partially supported by KAKENHI grant No.\ 20K03596 and 21H00986 
from the Japan Society for the Promotion of Science.  The second author is supported by a grant from the Simons Foundation (\#708778).


\begin{thebibliography}{99}

\bibitem{agol}I.\ Agol, 
{\em The classification of non-free 2-parabolic generator Kleinian groups}, 
Slides of talks given at Austin AMS Meeting and Budapest Bolyai conference, July 2002,
Budapest, Hungary.

\bibitem{agolliu}I.\ Agol and Y.\ Liu, 
{\em Presentation length and Simon's conjecture}, 
J. Amer. Math. Soc. {\bf 25} (2012), 151--187.

\bibitem{ALS}S.\ Aimi, D.\ Lee and M.\ Sakuma, 
{\em Parabolic generating pairs of $2$-bridge link groups}, 
in preparation. 

\bibitem{BBRW}M.\ Boileau, S.\ Boyer, A.\ Reid and S.\ Wang, 
{\em Simon's conjecture for two-bridge knots}, 
Comm. Anal. Geom. {\bf 18} (2010), 121--143.

\bibitem{ernstsumners}C.\ Ernst and D.\ Sumners,
{\em The growth of the number of prime knots}, 
Math. Proc. Cambridge Philos. Soc. {\bf 102} (1987), 303--315.

\bibitem{GonzalezAcunaRaminez}F.\ Gonz\'alez-Acu\~na and A.\ Ram\'inez, 
{\em Epimorphisms of knot groups onto free products}, 
Topology {\bf 42} (2003), 1205--1227. 

\bibitem{ORS}T.\ Ohtsuki, R.\ Riley and M.\ Sakuma,
{\em Epimorphisms between 2-bridge link groups},
Geom. Topol. Monogr. {\bf 14} (2008), 417--450.

\bibitem{NST}F.\ Nagasato, M.\ Suzuki and A.\ Tran, 
{\em On minimality of two-bridge knots}, 
Internat. J. Math. {\bf 28} (2017), 11 pages.

\bibitem{SZK}M.\ Suzuki, 
{\em Epimorphisms between $2$-bridge knot groups and their crossing numbers}, 
Algebraic and Geometric Topology 17 (2017), 2413--2428.

\bibitem{ST1}M.\ Suzuki and A.\ Tran, 
{\em Genera of two-bridge knots and epimorphisms of their knot groups}, 
Topology and its Applications 242 (2018), 66--72.

\bibitem{ST2}M.\ Suzuki and A.\ Tran, 
{\em Genera and crossing numbers of $2$-bridge knots}, 
To appear in Fundamenta Mathematicae.


\end{thebibliography}
\end{document}